\RequirePackage{fix-cm}
\documentclass[11pt, reqno]{amsart}
\usepackage{amsmath}
\usepackage{setspace}
\usepackage{fancyhdr}
\usepackage{mathrsfs}
\usepackage{xifthen}
\usepackage{verbatim}
\usepackage{hyperref}
\usepackage[all]{xcolor}

\usepackage{float}

\usepackage{comment}

\usepackage{geometry}
\usepackage{fancyhdr}
\usepackage{textcomp}
\usepackage{amsthm}
\usepackage{amstext}
\usepackage{amsfonts}
\usepackage{amssymb}
\usepackage{fixltx2e}
\usepackage{graphicx}

\usepackage{tikz}
\usepackage{tikz-cd}
\usepackage{caption}
\usetikzlibrary{calc,backgrounds}
\usetikzlibrary{matrix,arrows,decorations.pathmorphing}

\usepackage{bigstrut}

\usepackage{bbm}

\makeatletter

\newcommand{\schff}{X}
\newcommand{\adicff}{\mathcal{X}}
\newcommand{\Spa}{\mathrm{Spa}}

\newcommand{\slope}{\lambda}

\newcommand{\inj}{\hookrightarrow}
\newcommand{\surj}{\twoheadrightarrow}

\newcommand{\strsheaf}[1][]{\Ocal_{#1}}

\newcommand{\completion}[1]{\widehat{#1}}

\newcommand{\genlift}{\widetilde}
\newcommand{\genred}{\overline}

\newcommand{\gentuple}{\eta}
\newcommand{\genslope}{\lambda}

\newcommand{\dualpolygon}[1]{{#1}^*}
\newcommand{\lineseg}[2]{\underline{#1}^{(#2)}}
\newcommand{\HNslope}[2]{\slope_{#2}(#1)}

\newcommand{\slopedom}{\preceq}

\newcommand{\genvb}{\Ecal}
\newcommand{\gentorsor}{\Ecal}

\newcommand{\modification}[1]{{#1}'}
\newcommand{\gentrivialization}{\beta}

\newcommand{\punctuation}[1]{\widetilde{#1}}

\newcommand{\uniformizer}{\pi}
\newcommand{\pseudounif}{\varpi}
\newcommand{\genlocfield}{E}
\newcommand{\compmaxunram}[1]{\breve{#1}}

\newcommand{\FFring}{B}

\newcommand{\FFuntilt}{C}
\newcommand{\FFclosedpt}{\infty}
\newcommand{\genperfdring}{R}
\newcommand{\genpadicfield}{K}
\newcommand{\genperfalg}{A}
\newcommand{\integerring}[1]{\strsheaf[#1]}
\newcommand{\tilt}[1]{{#1}^\flat}
\newcommand{\inttilt}[1]{{#1}^{+\flat}}

\newcommand{\witt}{W}
\newcommand{\genfrob}{\phi}
\newcommand{\frobfieldaut}{\sigma}
\newcommand{\absgalgrp}{\Gamma}

\newcommand{\genabgrp}{M}
\newcommand{\gengrp}{H}

\newcommand{\topspace}[1]{\vert #1 \vert}

\newcommand{\BdR}{B_{\dR}}
\newcommand{\Be}{B_e}
\newcommand{\gendRunif}{t}
\newcommand{\fontainemap}{\theta}

\newcommand{\Newt}{\mathrm{Newt}}
\newcommand{\newtoncochar}[1]{\nu_{#1}}
\newcommand{\newtonmap}[1][]{%
\ifthenelse{\isempty{#1}}{{\nu}}{\nu(#1)}}%

\newcommand{\kottwitzmap}[1][]{%
\ifthenelse{\isempty{#1}}{{\kappa}}{\kappa(#1)}}%
\newcommand{\newtonset}[1]{\Ncal(#1)}

\newcommand{\newtonmaplevipart}[1]{\newtonmap^{(#1)}}
\newcommand{\newtonmaplevipartdual}[1]{\widehat{\newtonmap}^{(#1)}}
\newcommand{\newtonmaplevipartmid}{\omega}

\newcommand{\charlatt}{X^*(T)}
\newcommand{\cocharlatt}{X_*(T)}
\newcommand{\roots}{\Phi}
\newcommand{\coroots}{\Phi^\vee}

\newcommand{\gencocharacter}{\mu}
\newcommand{\genmaxtorus}{T}

\newcommand{\genborel}{B}

\newcommand{\genlevi}{M}
\newcommand{\weylgrp}{\Omega}

\newcommand{\levilift}{\dot}

\newcommand{\idcochar}{\mathbbm{1}}
\newcommand{\ordcochar}{\gencocharacter_{\mathrm{ord}}}
\newcommand{\simcochar}{\psi}

\newcommand{\genpartition}{\alpha}

\newcommand{\galavg}[1]{{#1}^\Diamond}
\newcommand{\cochardeg}[1]{{#1}^\sharp}

\numberwithin{equation}{section}

\usepackage{color}

\newcommand{\F}{\mathbb{F}}
\newcommand{\Q}{\mathbb{Q}}
\newcommand{\Z}{\mathbb{Z}}

\newcommand{\Gm}{\mathbb{G}_{\mathrm{m}}}

\newcommand{\et}{\text{\'{e}t}}

\DeclareMathOperator{\GL}{GL}

\DeclareMathOperator{\GSp}{GSp}

\DeclareMathOperator{\Spec}{Spec\,}

\DeclareMathOperator{\Bun}{Bun}

\DeclareMathOperator{\Gr}{Gr}

\DeclareMathOperator{\Proj}{Proj\,}

\DeclareMathOperator{\dR}{dR}

\DeclareFontFamily{OT1}{rsfs}{}
\DeclareFontShape{OT1}{rsfs}{n}{it}{<-> rsfs10}{}
\DeclareMathAlphabet{\mathscr}{OT1}{rsfs}{n}{it}

\newcommand{\Ecal}{\mathcal{E}}

\newcommand{\Ncal}{\mathcal{N}}
\newcommand{\Ocal}{\mathcal{O}}

\newcommand{\Ycal}{\mathcal{Y}}

\newtheorem{lemma}[subsection]{Lemma}
\newtheorem{prop}[subsection]{Proposition}
\newtheorem{cor}[subsection]{Corollary}

\theoremstyle{remark}
\newtheorem*{remark}{Remark}
\newtheorem{defn}[subsection]{Definition}
\newtheorem{example}[subsection]{Example}

\newtheorem*{thm*}{Theorem}
\makeatletter
\def\th@remark{%
  \thm@headfont{\bfseries}%
  \normalfont 
}
\def\imod#1{\allowbreak\mkern5mu({\operator@font mod}\,\,#1)}
\makeatother

\usepackage{enumitem}		
\theoremstyle{theorem}
\newtheorem{theorem}[subsection]{Theorem}

\numberwithin{equation}{section}

\usetikzlibrary{decorations.markings}
\tikzset{double line with arrow/.style args={#1,#2}{decorate,decoration={markings,%
mark=at position 0 with {\coordinate (ta-base-1) at (0,1pt);
\coordinate (ta-base-2) at (0,-1pt);},
mark=at position 1 with {\draw[#1] (ta-base-1) -- (0,1pt);
\draw[#2] (ta-base-2) -- (0,-1pt);
}}}}

\makeatother

\begin{document}
	
	\tikzset{
		node style sp/.style={draw,circle,minimum size=\myunit},
		node style ge/.style={circle,minimum size=\myunit},
		arrow style mul/.style={draw,sloped,midway,fill=white},
		arrow style plus/.style={midway,sloped,fill=white},
	}
    
	\title{On nonemptiness of Newton strata in the $B_{\dR}^+$-Grassmannian for $\GSp_{2n}$}
   
    \author[S. Hong]{Serin Hong}
    \address{Simons Laufer Mathematical Sciences Institute, 17 Gauss Way, Berkeley, CA 94720}
    \email{serinh@umich.edu}
    
    \begin{abstract} We build upon our previous work \cite{Hong_generalnewtonstrataGLn} to study the Newton stratification on the $\BdR^+$-Grassmannian for $\GSp_{2n}$. Our main result gives an explicit classification of all nonempty Newton strata associated to a Frobenius-conjugacy class satisfying certain conditions on the Newton polygon. 
    \end{abstract}
	
	\maketitle

	\tableofcontents
	
	\rhead{}

	\chead{}
\section{Introduction}

The $\BdR^+$-Grassmannian is a $p$-adic analogue of the complex affine Grassmannian. It has played a fundamental role in a number of major developments in $p$-adic geometry, including the work of Caraiani-Scholze \cite{CS_cohomunitaryshimura} on the cohomology of unitary Shimura varieties, the work of Scholze-Weinstein \cite{SW_berkeley} on the general construction of local Shimura varieties, and the work of Fargues-Scholze \cite{FS_geomLL} on the geometrization of the local Langlands correspondence. It has also been used as a central tool for studying the $p$-adic period domain by many authors, such as Chen-Fargues-Shen \cite{CFS_admlocus}, Shen \cite{Shen_HNstrata}, Chen \cite{Chen_FRconjnonbasic}, Viehmann \cite{Viehmann_weakadmlocNewton}, Nguyen-Viehmann \cite{NV_HNstrata}, and Chen-Tong \cite{CT_weakaddandnewton}.

In order to describe the geometry of the $\BdR^+$-Grassmannian, we consider a natural stratification called the \emph{Newton stratification}. 
Let $G$ be a connected reductive group over a finite extension $\genlocfield$ of $\Q_p$. We denote by $\Gr_G$ the $\BdR^+$-Grassmannian associated to $G$, and by $\Gr_{G, \gencocharacter}$ the Schubert cell associated to a dominant cocharacter $\gencocharacter$ of $G$. In addition, we write $\schff = \schff_\genlocfield$ for the Fargues-Fontaine curve over $\genlocfield$ and $\Bun_G$ for the stack of $G$-bundles on $\schff$. The work of Fargues \cite{Fargues_Gbundle} gives a natural bijection between the topological space $\topspace{\Bun_G}$ of $\Bun_G$ and the set $B(G)$ of Frobenius-conjugacy classes of $G(\compmaxunram{\genlocfield})$, where $\compmaxunram{\genlocfield}$ denotes the $p$-adic completion of the maximal unramified extension of $\genlocfield$. If we fix an element $b \in B(G)$ and a complete algebraically closed extension $\FFuntilt$ of $\genlocfield$, we get a natural map
\[ \Newt_b: \Gr_G(\FFuntilt) \longrightarrow |\Bun_G| \simeq B(G)\]
induced by the theorem of Beauville-Laszlo \cite{BL_beauvillelaszlothm}.
For each Schubert cell $\Gr_{G, \gencocharacter}$, the Newton stratification associated to $b$ is a decomposition
\[ \Gr_{G, \gencocharacter} = \bigsqcup_{\modification{b} \in B(G)} \Gr_{G, \gencocharacter, b}^{\modification{b}}\]
where $\Gr_{G, \gencocharacter, b}^{\modification{b}}(\FFuntilt)$ is the preimage of $\modification{b}$ in $\Gr_{G, \gencocharacter}(\FFuntilt)$ under the map $\Newt_b$.

In this paper, we study the Newton stratification of minuscule Schubert cells in the $\BdR^+$-Grassmannian for $\GSp_{2n}$. For a precise statement of our main result, let us recall some basic facts regarding the group $\GSp_{2n}$. As observed by Kottwitz \cite{Kottwitz_isocrystal}, an element in $B(\GSp_{2n})$ is determined by an invariant called the \emph{Newton polygon}, which is a concave polygon on the interval $[0, 2n]$ with integer breakpoints. In addition, we may represent each dominant cocharacter of $\GSp_{2n}$ as a $2n$-tuple of descending integers, or equivalently as a concave polygon on the interval $[0, 2n]$ where the slope on $[i-1, i]$ is given by the $i$-th entry of the $2n$-tuple. After multiplication by a central cocharacter, a minuscule dominant cocharacter of $\GSp_{2n}$ is either trivial or represented by the tuple $(\lineseg{1}{n}, \lineseg{0}{n})$, where we write $\lineseg{d}{n} := (d, \cdots, d) \in \Z^n$ for each $d \in \Z$. Since multiplication by a central cocharacter induces a natural bijection between Newton strata on Schubert cells (as noted in Proposition \ref{reduction to minimal minuscule}), it suffices to consider the minuscule cocharacter represented by $(1^{(n)}, 0^{(n)})$.

\begin{theorem}\label{classification of nonempty Newton strata for GSp2n, intro}
Let $b$ and $\modification{b}$ be elements in $B(\GSp_{2n})$ with their Newton polygons respectively denoted by $\newtonmap[b]$ and $\newtonmap[\modification{b}]$. Assume that the difference between any two distinct slopes in $\newtonmap[b]$ is greater than $1$. For the cocharacter $\gencocharacter$ of $\GSp_{2n}$ represented by the tuple $(\lineseg{1}{n}, \lineseg{0}{n})$, the Newton stratum $\Gr_{\GSp_{2n}, \gencocharacter, b}^{\modification{b}}$ is nonempty if and only if the following conditions are satisfied:
\begin{enumerate}[label=(\roman*)]
\item\label{modified Kottwitz set condition for nonempty Newton strata, intro} The polygon $\newtonmap[\modification{b}]$ lies below the polygon $\newtonmap[b] + \dualpolygon{\gencocharacter}$ with the same endpoints, where $\dualpolygon{\gencocharacter}$ denotes the cocharacter of $\GSp_{2n}$ represented by the tuple $(\lineseg{0}{n}, \lineseg{-1}{n})$. 
\smallskip

\item\label{slopewise dominance condition for nonemtpy Newton strata, intro} For $i = 1, 2, \cdots, 2n$, we have an inequality
\[\HNslope{\newtonmap[\modification{b}]}{i} \leq \HNslope{\newtonmap[b]}{i} \leq \HNslope{\newtonmap[\modification{b}]}{i} +1\]
where $\HNslope{\newtonmap[b]}{i}$ and $\HNslope{\newtonmap[\modification{b}]}{i}$ respectively denote the slopes of $\newtonmap[b]$ and $\newtonmap[\modification{b}]$ on $[i-1, i]$. 
\vspace{0.01in}	
\item\label{breakpoint condition for nonempty Newton strata, intro} For each breakpoint of $\newtonmap[b]$, there exists 
a breakpoint of $\newtonmap[\modification{b}]$ with the same $x$-coordinate. 
\end{enumerate}
\end{theorem}
\begin{figure}[H]
\begin{tikzpicture}[scale=0.7]
		\coordinate (p00) at (1.5, 4);
		\coordinate (p11) at (2.8, 6.3);
		\coordinate (r11) at (4, 7.8);
		\coordinate (p22) at (5.2, 8.7);
		\coordinate (p33) at (6.5, 9);
		\coordinate (p44) at (8, 8);

		\coordinate (left) at (0, 0);
		\coordinate (q0) at (1.5, 3.2);
		\coordinate (q1) at (4, 5.2);
		\coordinate (q2) at (6.5, 5.7);
		\coordinate (q3) at (8, 4);
		
		\coordinate (r1) at (4, 3.8);
		\coordinate (r2) at (6.5, 3.5);

		\coordinate (p0) at (1.5, 2.5);
		\coordinate (p1) at (2.8, 3.5);
		\coordinate (p2) at (5.2, 3.5);
		\coordinate (p3) at (6.5, 2.5);
		\coordinate (p4) at (8, 0);
				
		\draw[step=1cm,thick, color=red] (left) -- (q0) --  (q1) -- (q2) -- (q3);
		\draw[step=1cm,thick, color=green] (left) -- (p0) --  (p1) -- (r1) -- (p2) -- (p3) -- (p4);
		\draw[step=1cm,thick, color=teal] (left) -- (p00) --  (p11) -- (r11) -- (p22) -- (p33) -- (p44);
		\draw[step=1cm,thick, color=blue] (q1) -- (r2) -- (p4);
		
		\draw [fill] (q0) circle [radius=0.05];		
		\draw [fill] (q1) circle [radius=0.05];		
		\draw [fill] (q2) circle [radius=0.05];		
		\draw [fill] (q3) circle [radius=0.05];		
		\draw [fill] (left) circle [radius=0.05];
		
		\draw [fill] (r1) circle [radius=0.05];		
		\draw [fill] (r2) circle [radius=0.05];	
		
		\draw [fill] (p0) circle [radius=0.05];		
		\draw [fill] (p1) circle [radius=0.05];		
		\draw [fill] (p2) circle [radius=0.05];		
		\draw [fill] (p3) circle [radius=0.05];		
		\draw [fill] (p4) circle [radius=0.05];		

		\draw [fill] (p00) circle [radius=0.05];		
		\draw [fill] (p11) circle [radius=0.05];		
		\draw [fill] (r11) circle [radius=0.05];		
		\draw [fill] (p22) circle [radius=0.05];	
		\draw [fill] (p33) circle [radius=0.05];	
		\draw [fill] (p44) circle [radius=0.05];	
		
		\draw[step=1cm,dotted] (1.5, -0.4) -- (1.5, 9.2);
		\draw[step=1cm,dotted] (4, -0.4) -- (4, 9.2);
		\draw[step=1cm,dotted] (6.5, -0.4) -- (6.5, 9.2);

		
		\path (q1) ++(0.8, 0.6) node {\color{red}$\newtonmap[b]$};
		\path (r2) ++(1.8, -0.8) node {\color{blue}$\newtonmap[b] + \dualpolygon{\gencocharacter}$};
		\path (p00) ++(-1.8, 0) node {\color{teal}$\newtonmap[\modification{b}] + \mathrm{id}_{[0, 2n]}$};
		\path (p0) ++(0.7, -0.3) node {\color{green}$\newtonmap[\modification{b}]$};

\end{tikzpicture}
\caption{Illustration of the conditions in Theorem \ref{classification of nonempty Newton strata for GSp2n, intro}}
\end{figure}

Let us briefly sketch our proof of Theorem \ref{classification of nonempty Newton strata for GSp2n, intro}. We consider the natural embedding $\GSp_{2n} \inj \GL_{2n}$ and identify $b, \modification{b} \in B(\GSp_{2n})$ with their images under the induced map $B(\GSp_{2n}) \longrightarrow B(\GL_{2n})$. Then we obtain a natural map
\[ \Gr_{\GSp_{2n}, \gencocharacter, b}^{\modification{b}} \longrightarrow \Gr_{\GL_{2n}, \gencocharacter, b}^{\modification{b}},\]
and consequently deduce the necessity part of Theorem \ref{classification of nonempty Newton strata for GSp2n, intro} from a previous result of the author \cite{Hong_generalnewtonstrataGLn}. For the sufficiency part of Theorem \ref{classification of nonempty Newton strata for GSp2n, intro}, we regard the Newton polygons $\newtonmap[b]$ and $\newtonmap[\modification{b}]$ as rational cocharacters of $\GSp_{2n}$. Let $\genlevi$ denote the centralizer of $\newtonmap[b]$, which is a Levi subgroup of $\GSp_{2n}$. By construction, the element $b \in B(\GSp_{2n})$ admits a \emph{reduction} $b_\genlevi \in B(\genlevi)$; in other words, $b$ is the image of $b_\genlevi$ under the natural map $B(\genlevi) \longrightarrow B(\GSp_{2n})$. Moreover, the condition \ref{breakpoint condition for nonempty Newton strata, intro} implies that $\modification{b} \in B(\GSp_{2n})$ admits a reduction $\modification{b}_\genlevi \in B(\genlevi)$. Hence we obtain a natural map
\[ \Gr_{\genlevi, \gencocharacter, b_\genlevi}^{\modification{b}_\genlevi} \longrightarrow \Gr_{\GSp_{2n}, \gencocharacter, b}^{\modification{b}}.\]
Let us identify the Newton polygons of $b_\genlevi$ and $\modification{b}_\genlevi$ respectively with $\newtonmap[b]$ and $\newtonmap[\modification{b}]$. Since the Newton polygon of $b_\genlevi$ is central in $\genlevi$, we deduce the nonemptiness of $\Gr_{\genlevi, \gencocharacter, b_\genlevi}^{\modification{b}_\genlevi}$ from the conditions in Theorem \ref{classification of nonempty Newton strata for GSp2n, intro} by the work of Rapoport \cite{Rapoport_basicminuscule}, Chen-Fargues-Shen \cite{CFS_admlocus}, Shen \cite{Shen_HNstrata}, and Viehmann \cite{Viehmann_weakadmlocNewton}, thereby establishing the nonemptiness of $\Gr_{\GSp_{2n}, \gencocharacter, b}^{\modification{b}}$.

We expect that our strategy can be used to establish an analogous statement of Theorem \ref{classification of nonempty Newton strata for GSp2n, intro} for other classical groups. We plan to update this manuscript as we make progress on other groups.

\subsection*{Acknowledgments} 
The author would like to thank Linus Hamann and Tasho Kaletha for interesting discussions. 
This work was partially supported by the Simons Foundation under Grant Number 814268.

\section{Notations and conventions}\label{notations}

Throughout the paper, we consider the following data:

\begin{itemize}
\item $\genlocfield$ is a finite extension of $\Q_p$ with uniformizer $\uniformizer$. 
\smallskip

\item $\FFuntilt$ is a complete and algebraically closed extension of $\genlocfield$. 
\smallskip

\item $G$ is a reductive group over $\genlocfield$ with Borel subgroup $\genborel$ and maximal torus $\genmaxtorus \subseteq \genborel$. 
\end{itemize}
Given such data, we set up the following notations:
\begin{itemize}
\item $\absgalgrp$ is the absolute Galois group of $\genlocfield$. 
\smallskip

\item $\compmaxunram{\genlocfield}$ is the $p$-adic completion of the maximal unramified extension of $\genlocfield$.
\smallskip

\item $\frobfieldaut$ is the Frobenius automorphism of $\compmaxunram{\genlocfield}$
\smallskip

\item $B(G)$ is the set of $\frobfieldaut$-conjugacy classes of $G(\compmaxunram{\genlocfield})$. 
\smallskip

\item $(\charlatt, \roots, \cocharlatt, \coroots)$ is the absolute root datum of $G$.
\smallskip

\item $\weylgrp$ is the Weyl group of $G$. 
\end{itemize}
In addition, we use the following general notations:
\begin{itemize}
\item Given a valued field $\genpadicfield$, we denote its valuation ring by $\integerring{\genpadicfield}$. 
\smallskip


\smallskip

\item Given a perfectoid ring $\genperfdring$, we denote its tilt by $\tilt{\genperfdring}$.
\smallskip

\item Given a perfect $\F_p$-algebra $\genperfalg$, we write $\witt(\genperfalg)$ for the ring of Witt vectors over $\genperfalg$. 
\smallskip

\item Given an abelian group $\genabgrp$ equipped with an action of a group $\gengrp$, we write $\genabgrp^\gengrp$ and $\genabgrp_\gengrp$ respectively for the subgroups of $\gengrp$-invariants and $\gengrp$-coinvariants. 
\end{itemize}

\section{$G$-bundles on the Fargues-Fontaine curve}\label{G-bundles on FF curve}

\begin{defn}\label{FF curve} Let us fix a pseudouniformizer $\pseudounif$ of $\tilt{\FFuntilt}$. In addition, we write $q$ for the number of elements in the residue field of $\genlocfield$ and set $\witt_{\integerring{\genlocfield}}(\integerring{\tilt{\FFuntilt}}): = \witt(\integerring{\tilt{\FFuntilt}}) \otimes_{\witt(\F_q)} \integerring{\genlocfield}$.

\begin{enumerate}[label=(\arabic*)]
\item We define the \emph{perfectoid punctured unit disk} associated to the pair $(\genlocfield, \tilt{\FFuntilt})$ by
\[ \Ycal:= \Spa(\witt_{\integerring{\genlocfield}}(\integerring{\tilt{\FFuntilt}})\setminus\{|\uniformizer [\pseudounif]|=0\}\]
where $[\pseudounif]$ denotes the Teichmuller lift of $\pseudounif$ in $\witt_{\integerring{\genlocfield}}(\integerring{\tilt{\FFuntilt}})$.
\smallskip

\item We write $\genfrob$ for the automorphism of $\Ycal$ induced by the $q$-Frobenius automorphism on $\witt_{\integerring{\genlocfield}}(\integerring{\tilt{\FFuntilt}})$, and define the \emph{adic Fargues-Fontaine curve} associated to the pair $(\genlocfield, \tilt{\FFuntilt})$ by
\[ \adicff:= \Ycal/\genfrob^\Z.\]

\item We define the \emph{schematic Fargues-Fontaine curve} associated to the pair $(\genlocfield, \tilt{\FFuntilt})$ by
\[
\schff:= \Proj \left( \bigoplus_{n \geq 0 } \FFring^{\genfrob = \uniformizer^n} \right)
\]
where $\FFring$ denotes the ring of global sections on $\Ycal$. 
\end{enumerate}
\end{defn}

\begin{remark}
The construction of the adic Fargues-Fontaine curve relies on the fact that the action of $\genfrob$ on $\Ycal$ is properly discontinuous. 
\end{remark}

\begin{theorem}[{\cite[Theorem 4.10]{Kedlaya_noeth}}, {\cite[Th\'eor\`eme 6.5.2
]{FF_curve}}, {\cite[Theorem 8.7.7]{KL_relpadic1}}]\label{geom properties of FF curve} 
We have the following facts regarding $\adicff$ and $\schff$:
\begin{enumerate}[label=(\arabic*)]
\item $\adicff$ is a Noetherian adic space over $\genlocfield$. 

\item $\schff$ is a Dedekind scheme over $\genlocfield$. 
\smallskip

\item\label{GAGA for FF curve} The the categories of vector bundles on $\adicff$ and $\schff$ are equivalent via the pullback along a natural map $\adicff \longrightarrow \schff$ of locally ringed spaces.
\end{enumerate}
\end{theorem}

\begin{remark}
The statement \ref{GAGA for FF curve} in Theorem \ref{geom properties of FF curve} allows us to identify $G$-bundles on $\adicff$ with $G$-bundles on $\schff$. 
\end{remark}

\begin{defn}
%
%
Given an element $b \in B(G)$, we define the associated $G$-bundle $\genvb_b$ on $\adicff$ (or on $\schff)$ to be the descent along the map $\Ycal \longrightarrow \Ycal/\genfrob^\Z = \adicff$ of the trial $G$-bundle on $\Ycal$ equipped with the $\genfrob$-linear automorphism induced by $b$. 
\end{defn}


\begin{theorem}[{\cite[Th\'eor\`eme 5.1]{Fargues_Gbundle}}]\label{classification of G-bundles on FF curve}
There is a natural bijection 
\[ B(G) \stackrel{\sim}{\longrightarrow} H_{\et}^1(\adicff, G)\]
which maps each $b \in B(G)$ to the isomorphism class of $\genvb_b$. 
\end{theorem}

\begin{defn}
We define the \emph{Newton set} of $G$ by
\[\newtonset{G}:= (\cocharlatt_\Q/\weylgrp)^\absgalgrp,\]
where we write $\cocharlatt_\Q:= \cocharlatt \otimes_{\Z} \Q$ for the space of rational cocharacters of $G$, and the \emph{algebraic fundamental group} of $G$ by
\[\pi_1(G):= \cocharlatt/\langle\coroots\rangle,\]
where $\langle\coroots\rangle$ denotes the subgroup of $\cocharlatt$ generated by $\coroots$. 
%
%
\end{defn}

\begin{prop}[{\cite[\S4]{Kottwitz_isocrystal}, \cite[\S6]{Kottwitz_isocrystal2}}]\label{newton map and kottwitz map}
There exist unique natural transformations
\[\newtonmap: B(\cdot) \longrightarrow \newtonset{\cdot} \quad \text{ and } \quad \kottwitzmap: B(\cdot) \longrightarrow \pi_1(\cdot)_\absgalgrp\]
of set-valued functors on the category of connected reductive groups over $\genlocfield$ such that the following statements hold:
\begin{enumerate}[label=(\arabic*)]
\item Under the map $\newtonmap: B(G) \longrightarrow \newtonset{G}$, the image of $b \in B(G)$ with a representative $\tilde{b}$ is given by a unique rational cocharacter $\newtoncochar{\tilde{b}}$ of $G$ for which there exist an integer $m>0$ and an element $c \in G(\compmaxunram{\genlocfield})$ with the following properties:
\smallskip
\begin{enumerate}[label=(\roman*)]
\item $m\newtoncochar{\tilde{b}}$ lies in $\cocharlatt$. 
\smallskip

\item $\mathrm{Int}(c) \circ m\newtoncochar{\tilde{b}}$ is defined over the fixed field of $\frobfieldaut^m$ in $\compmaxunram{\genlocfield}$, where $\mathrm{Int}(c)$ denotes the inner automorphism mapping each $g \in G(\compmaxunram{\genlocfield})$ to $cgc^{-1}$. 
\smallskip

\item We have $c \cdot \tilde{b} \cdot \frobfieldaut(\tilde{b}) \cdots \frobfieldaut^m(\tilde{b}) \cdot \frobfieldaut^m(c)^{-1} = c \cdot (m\newtoncochar{\tilde{b}})(\uniformizer) \cdot c^{-1}$. 
\end{enumerate}
\smallskip

\item We have a commutative diagram
\[
\begin{tikzcd}
\compmaxunram{\genlocfield}^\times \arrow[d, twoheadrightarrow] \arrow[r, "\mathrm{val}"] & \Z \arrow[d, -, double line with arrow={-,-}]\\
B(\Gm) \arrow[r, "\kottwitzmap"] & \pi_1(\Gm)_\absgalgrp
\end{tikzcd}
\]
where $\mathrm{val}: \compmaxunram{\genlocfield}^\times  \longrightarrow \Z$ is the $\uniformizer$-adic valuation map on $\compmaxunram{\genlocfield}$. 
\end{enumerate}
\end{prop}

\begin{remark}
If $G$ is unramified, there is an explicit description of the map $\kottwitzmap: B(G) \longrightarrow \pi_1(G)_\absgalgrp$ given by Rapoport-Viehmann \cite[Remark 2.3]{RV_localshimura}. 
\end{remark}

\begin{defn}
We refer to the maps $\newtonmap : B(G) \longrightarrow \newtonset{G}$ and $\kottwitzmap: B(G) \longrightarrow \pi_1(G)_\absgalgrp$ from Proposition \ref{newton map and kottwitz map} respectively as the \emph{Newton map} and the \emph{Kottwitz map} of $G$. 
\end{defn}

\begin{remark}
In line with Theorem \ref{classification of G-bundles on FF curve}, Fargues \cite[Propositions 6.6 and 8.1]{Fargues_Gbundle} provides a geometric description of $\newtonmap$ and $\kottwitzmap$ in terms of $G$-bundles on the Fargues-Fontaine curve. We will briefly sketch this description for $G = \GL_n$ after Example \ref{examples of newton map and kottwitz map}. 
\end{remark}

\begin{prop}[{\cite[\S4.13]{Kottwitz_isocrystal2}}]\label{classification of B(G) via newton map and kottwitz map}
Every element of $B(G)$ is uniquely determined by its image under the Newton map and the Kottwitz map. In other words, the map
\[ \newtonmap \times \kottwitzmap : B(G) \longrightarrow \newtonset{G} \times \pi_1(G)_\absgalgrp\]
is injective. 
\end{prop}

\begin{prop}[{\cite[Theorem 1.15]{RR_isocrystals}, \cite[\S2]{RV_localshimura}}]\label{comm diagram for newton map and kottwitz map}
The maps $\newtonmap$ and $\kottwitzmap$ fit into a natural commutative diagram
\[
\begin{tikzcd}
B(G) \arrow[d, "\newtonmap", swap] \arrow[rr, "\kottwitzmap"] & & \pi_1(G)_\absgalgrp \arrow[d] \\[0.5em]
 \newtonset{G} \arrow[r, hookrightarrow] & \cocharlatt_\Q \arrow[r, twoheadrightarrow] & \pi_1(G)_\absgalgrp \otimes_{\Z} \Q
\end{tikzcd}
\]
where the inclusion $\newtonset{G} \inj \cocharlatt_\Q$ is obtained by identifying $\cocharlatt_\Q/\weylgrp$ with the set of dominant rational cocharacters. 
\end{prop}

\begin{cor}\label{injectivity of newton map and torsion on fund grp}
If $\pi_1(G)_\absgalgrp$ is torsion free, the Newton map of $G$ is injective. 
\end{cor}

\begin{proof}
When $\pi_1(G)_\absgalgrp$ is torsion free, the natural map $\pi_1(G)_\absgalgrp \longrightarrow \pi_1(G)_\absgalgrp \otimes_\Z \Q$ is injective. Hence the assertion follows from Propositions \ref{classification of B(G) via newton map and kottwitz map} and \ref{comm diagram for newton map and kottwitz map}. 
\end{proof}

\begin{defn}
We say that an element $b \in B(G)$ is \emph{basic} if $\newtonmap[b]$ is represented by a central rational cocharacter. 
\end{defn}

\begin{prop}\label{kottwitz map for basic elements}
Let $B(G)_\mathrm{basic}$ denote the set of basic elements in $B(G)$. The Kottwitz map induces a bijection
\[ B(G)_\mathrm{basic} \simeq \pi_1(G)_\absgalgrp.\]
\end{prop}

\begin{example}\label{examples of newton map and kottwitz map}
Let us discuss the Newton map and the Kottwitz map for some groups. 
\begin{enumerate}[label=(\arabic*)]
\item For $\GL_n$, we have natural identifications
\[ \newtonset{\GL_n} \cong \{(\genslope_1, \cdots, \genslope_n) \in \Q^n: \genslope_1 \geq \cdots \geq \genslope_n\} \quad \text{ and } \quad \pi_1(\GL_n)_\absgalgrp \cong \Z.\]
We regard elements of $\newtonset{\GL_n}$ as concave polygons on the interval $[0, n]$ with rational slopes. Then we have the following facts:
\smallskip
\begin{enumerate}[label=(\alph*)]
\item The image of the Newton map for $\GL_n$ consists of all concave polygons on $[0, n]$ with integer breakpoints, essentially by the classical result of Manin \cite{Manin_maninclassification}. 
\smallskip

\item An element $b \in B(\GL_n)$ is basic if and only if its image under the Newton map is a line segment. 
\smallskip

\item The Newton map of $\GL_n$ is injective by Corollary \ref{injectivity of newton map and torsion on fund grp}. 
\smallskip

\item Proposition \ref{comm diagram for newton map and kottwitz map} yields a commutative diagram
\begin{equation*}\label{comm diagram for newton map and kottwitz map of GLn}
\begin{tikzcd}
B(\GL_n) \arrow[d, "\newtonmap", swap] \arrow[r, "\kottwitzmap"] &[2em] \pi_1(\GL_n)_\absgalgrp \cong \Z \arrow[d, hookrightarrow] \\[0.5em]
 \newtonset{\GL_n} \arrow[r, "\sum"] & \Q
\end{tikzcd}
\end{equation*}
where $\sum$ maps each element in $\newtonset{\GL_n}$ to the sum of its entries. 
\end{enumerate}
\smallskip

\item For $\GSp_{2n}$, we consider the natural embedding into $\GL_{2n}$ and obtain an identification
\[ 
\newtonset{\GSp_{2n}} \cong \{(\genslope_1, \cdots, \genslope_{2n}) \in \newtonset{\GL_{2n}}: \genslope_i + \genslope_{2n+1-i} \text{ constant}\}.
\]
Moreover, we have $\pi_1(\GSp_{2n})_\absgalgrp \cong \Z$. 
Then we find the following facts:
\smallskip
\begin{enumerate}[label=(\alph*)]
\item An element $(\genslope_1, \cdots, \genslope_{2n}) \in \newtonset{\GSp_{2n}}$ lies in the image of the Newton map if and only if it has integer breakpoints with $\genslope_i + \genslope_{2n+1-i} \in \Z$ for $i = 1, \cdots, n$.  
\smallskip

\item An element $b \in B(\GSp_{2n})$ is basic if and only if its image under the Newton map is a line segment. 
\smallskip

\item The Newton map of $\GSp_{2n}$ is injective by Corollary \ref{injectivity of newton map and torsion on fund grp}. 
\smallskip

\item Proposition \ref{comm diagram for newton map and kottwitz map} yields a commutative diagram
\begin{equation*}\label{comm diagram for newton map and kottwitz map of GLn}
\begin{tikzcd}
B(\GSp_{2n}) \arrow[d, "\newtonmap", swap] \arrow[r, "\kottwitzmap"] &[2em] \pi_1(\GSp_{2n})_\absgalgrp \cong \Z \arrow[d, hookrightarrow] \\[0.5em]
 \newtonset{\GSp_{2n}} \arrow[r, "\frac{1}{n}\sum"] & \Q
\end{tikzcd}
\end{equation*}
where $\sum$ maps each element in $\newtonset{\GSp_{2n}}$ to the sum of its entries. 
\end{enumerate}
\end{enumerate}
\end{example}

\begin{remark}
Let us provide an explicit description of the Newton map and the Kottwitz map for $\GL_n$ in terms of vector bundles on $\schff$. 
By Theorem \ref{geom properties of FF curve}, the Harder-Narasimhan formalism applies to the category of vector bundles on $\schff$. Under the identifications for $\newtonset{\GL_n}$ and $\pi_1(\GL_n)_\absgalgrp$ in Example \ref{examples of newton map and kottwitz map}, $\newtonmap[b]$ and $\kottwitzmap[b]$ for each $b \in B(\GL_n)$ are respectively given by the Harder-Narasimhan polygon and degree of $\genvb_b$. 
In addition, an element $b \in B(\GL_n)$ is basic if and only if $\genvb_b$ is semistable. 
The injectivity of the Newton map for $\GL_n$ means that vector bundles on $\schff$ are determined by its Harder-Narasimhan polygon. 
\end{remark}

\section{The Newton stratification of the $\BdR^+$-Grassmannian}

\begin{prop}[{\cite[Proposition 2.4]{Fontaine_BdR}}, {\cite[Lemma 3.6.3]{KL_relpadic1}}]\label{relative fontaine map}
Every perfectoid affinoid algebra $(\genperfdring, \genperfdring^+)$ over $\genlocfield$ induces a natural surjective homomorphism $\fontainemap_{\genperfdring, \genperfdring^+}: \witt(\inttilt{\genperfdring}) \surj \genperfdring^+$ whose kernel is a principal ideal of $\witt(\inttilt{\genperfdring})$. 
\end{prop}

\begin{defn}
Let $(\genperfdring, \genperfdring^+)$ be a perfectoid affinoid algebra over $\genlocfield$. 
\begin{enumerate}[label=(\arabic*)]
\item We refer to the map $\fontainemap_{\genperfdring, \genperfdring^+}$ in Proposition \ref{relative fontaine map} as the \emph{Fontaine map} for $(\genperfdring, \genperfdring^+)$
\smallskip

\item We define the \emph{positive de Rham ring} associated to $\genperfdring$, denoted by $\BdR^+(\genperfdring)$, to be the completion of $\witt(\inttilt{\genperfdring})[1/p]$ with respect to the kernel of $\fontainemap_{\genperfdring, \genperfdring^+}$. 
\smallskip

\item We fix an element $\gendRunif_{\genperfdring, \genperfdring^+} \in \witt(\inttilt{\genperfdring})$ which generates the kernel of $\fontainemap_{\genperfdring, \genperfdring^+}$, and define the \emph{de Rham ring} associated to $\genperfdring$ by $\BdR(\genperfdring):= \BdR^+(\genperfdring)[1/\gendRunif_{\genperfdring, \genperfdring^+}]$.
\smallskip

\item For $\genperfdring = \FFuntilt$, we refer to $\BdR:= \BdR(\FFuntilt)$ and $\BdR^+:= \BdR^+(\FFuntilt)$ respectively as the \emph{de Rham period ring} and the \emph{positive de Rham period ring}, and write $\gendRunif:= \gendRunif_{\FFuntilt, \integerring{\FFuntilt}}$. 
\end{enumerate}
\end{defn}

\begin{remark}
As the notation suggests, the construction of $\BdR^+(\genperfdring)$ and $\BdR(\genperfdring)$ does not depend on $\genperfdring^+$ or $\gendRunif_{\genperfdring, \genperfdring^+}$. 
\end{remark}

\begin{prop}[{\cite[Proposition 2.17]{Fontaine_BdR}}]\label{BdR as a discrete valued field} 
The ring $\BdR$ is a discretely valued field with $\BdR^+$ as the valuation ring and $\FFuntilt$ as the residue field. 
\end{prop}




\begin{defn}
The \emph{$\BdR^+$-Grassmannian} is the functor $\Gr_G$ which associates to each perfectoid affinoid $\genlocfield$-algebra $(\genperfdring, \genperfdring^+)$ the set of pairs $(\gentorsor, \gentrivialization)$ consisting of a $G$-torsor $\gentorsor$ on $\Spec(\BdR^+(\genperfdring))$ and a trivialization $\gentrivialization$ of $\gentorsor$ over $\Spec(\BdR(\genperfdring))$. For simplicity, we write $\Gr_G(\FFuntilt):= \Gr_G(\FFuntilt, \integerring{\FFuntilt})$. 
\end{defn}


\begin{prop}[{\cite[Proposition 19.1.2]{SW_berkeley}}]\label{points of BdR-Grassmannian}
There exists a canonical identification 
\[\Gr_G(\FFuntilt) \cong G(\BdR)/G(\BdR^+).\]
\end{prop}


\begin{prop}[{\cite[Corollary 19.3.4]{SW_berkeley}}]\label{schubert cells as diamonds}
Let $\gencocharacter$ be a dominant cocharacter of $G$. There exists a locally spatial diamond $\Gr_{G, \gencocharacter}$ with
\[ \Gr_{G, \gencocharacter}(\FFuntilt) =  G(\BdR^+)\gencocharacter(\gendRunif)^{-1}G(\BdR^+)/G(\BdR^+).\]
\end{prop}

\begin{remark}
We will not use the language of diamonds in a significant way as we will be mainly interested in the $\FFuntilt$-valued points of $\Gr_G$ and $\Gr_{G, \gencocharacter}$.
\end{remark}

\begin{defn}
Given a dominant cocharacter $\gencocharacter$ of $G$, we refer to the locally spatial diamond $\Gr_{G, \gencocharacter}$ given by Proposition \ref{schubert cells as diamonds} as the \emph{Schubert cell} of $\Gr_G$ associated to $\gencocharacter$. 
\end{defn}

\begin{prop}[{\cite[Th\'eor\`emes 6.5.2 and 7.3.3]{FF_curve}}]\label{distinguished closed pt on FF curve}
There exists a closed point $\FFclosedpt$ on $\schff$ satisfying the following properties:
\begin{enumerate}[label=(\roman*)]
\item $\schff - \FFclosedpt$ is isomorphic to $\Spec(\Be)$ for some principal ideal domain $\Be \subseteq \BdR$. 
\smallskip

\item The completed local ring at $\FFclosedpt$ is naturally isomorphic to $\BdR^+$. 
\smallskip

\item Every $G$-bundle on $\schff$ pulls back to a trivial $G$-bundle via the map $\Spec(\BdR^+) \to \schff$ induced by $\FFclosedpt$. 
\end{enumerate}
\end{prop}

\begin{defn}
We fix a closed point $\FFclosedpt$ on $\schff$ given by Proposition \ref{distinguished closed pt on FF curve}, and refer to $\FFclosedpt$ as the \emph{distinguished closed point} on $\schff$. 
\end{defn}

\begin{remark}
Our terminology is justified by the fact that there is a canonical choice of $\FFclosedpt$. The key technical fact is that closed points on $\schff$ are naturally in bijection with the equivalence classes of untilts of $\tilt{\FFuntilt}$, as shown by Kedlaya-Liu \cite[Theorem 8.7.7]{KL_relpadic1}. Our canonical choice of $\FFclosedpt$ corresponds to the identity map on $\tilt{\FFuntilt}$.
\end{remark}

\begin{prop}\label{Beauville-Laszlo for FF curve}
The set $H_{\et}^1(\schff, G)$ is canonically in bijection with the set of isomorphism classes of triples $(\punctuation{\genvb}, \completion{\genvb}, \gentrivialization)$ where
\begin{itemize}
\item $\punctuation{\genvb}$ is a $G$-bundle on $\schff - \FFclosedpt$,
\smallskip

\item $\completion{\genvb}$ is a trivial $G$-bundle on $\Spec(\BdR^+)$, and
\smallskip

\item $\gentrivialization$ is a gluing map of $\punctuation{\genvb}$ and $\completion{\genvb}$ over $\Spec(\BdR)$. 
\end{itemize}
\end{prop}
\begin{proof}
The assertion is evident by the theorem of Beauville-Laszlo \cite{BL_beauvillelaszlothm}. 
\end{proof}

\begin{defn}\label{modification of G-bundle on FF curve}
Let $\genvb$ be a $G$-bundle on $\schff$. By a \emph{modification} of $\genvb$ at $\FFclosedpt$, we mean a $G$-bundle $\modification{\genvb}$ on $\schff$ together with an isomorphism between the restrictions of $\genvb$ and $\modification{\genvb}$ on $\schff - \FFclosedpt$. 
\end{defn}

\begin{example}\label{modification given by a point on BdR+ Grassmannian}
Let us consider an element $b \in B(G)$ and a point $x \in \Gr_G(\FFuntilt)$. We have $x = gG(\BdR^+)$ for some $g \in G(\BdR)$ under the identification $\Gr_G(\FFuntilt) \cong G(\BdR)/G(\BdR^+)$ given by Proposition \ref{points of BdR-Grassmannian}. In light of Proposition \ref{Beauville-Laszlo for FF curve} we take a triple $(\punctuation{\genvb}, \completion{\genvb}, \gentrivialization)$ corresponding to $\genvb_b$ and a $G$-bundle $\genvb_{b, x}$ on $\schff$ corresponding to $(\punctuation{\genvb}, \completion{\genvb}, g\gentrivialization)$. 
By construction, 
$\genvb_{b, x}$ is a modification of $\genvb_b$ at $\FFclosedpt$. 
\end{example}

\begin{remark}
It is not hard to see that the isomorphism type of $\genvb_{b, x}$ does not depend on the choice of $g \in G(\BdR)$.
\end{remark}

\begin{defn}\label{def of newton stratification}
Consider an element $b \in B(G)$ and a dominant cocharacter $\gencocharacter$ of $G$. 
\begin{enumerate}[label=(\arabic*)]
\item For every $x \in \Gr_G(\FFuntilt)$, we refer to the $G$-bundle $\genvb_{b, x}$ in Example \ref{modification given by a point on BdR+ Grassmannian} as the \emph{modification of $\genvb_b$ at $\FFclosedpt$ induced by $x$}. 
\smallskip

\item For every $\modification{b} \in B(G)$, we define the associated \emph{Newton stratum with respect to $b$} in $\Gr_{G, \gencocharacter}$ to be the subdiamond $\Gr_{G, \gencocharacter, b}^{\modification{b}}$ of $\Gr_{G, \gencocharacter}$ with
\[\Gr_{G, \gencocharacter, b}^{\modification{b}}(\FFuntilt) = \{x \in \Gr_{G, \gencocharacter}(\FFuntilt): \genvb_{b, x} \simeq \genvb_{\modification{b}}\}.\]
\end{enumerate}
\end{defn}

\begin{remark}
The subdiamond $\Gr_{G, \gencocharacter, b}^{\modification{b}}$ of $\Gr_{G, \gencocharacter}$ is uniquely determined by its set of $\FFuntilt$-points since $\Gr_{G, \gencocharacter}$ is a locally spatial diamond.
\end{remark}

\begin{prop}\label{nonemptiness of newton strata via points}
Let $\gencocharacter$ be a dominant cocharacter of $G$. Given two elements $b, \modification{b} \in B(G)$, the Newton stratum $\Gr_{G, \gencocharacter, b}^{\modification{b}}$ is not empty if and only if it has a $\FFuntilt$-valued point. 
\end{prop}

\begin{proof}
The assertion is evident by definition. 
\end{proof}

\begin{prop}\label{functoriality of newton strata}
Let $f: G \inj \genlift{G}$ be an embedding of reductive groups over $\genlocfield$. We take elements $b, \modification{b} \in B(G)$ and write $\genlift{b}, \genlift{\modification{b}}$ for their images under the map $B(G) \longrightarrow B(\genlift{G})$ induced by $f$. In addition, we consider a dominant cocharacter $\gencocharacter$ of $G$ and denote by $\genlift{\gencocharacter}$ the dominant cocharacter of $\genlift{G}$ associated to $f \circ \gencocharacter$. There exists a natural map
\[ \Gr_{G, \gencocharacter, b}^{\modification{b}}(\FFuntilt) \longrightarrow \Gr_{\genlift{G}, \genlift{\gencocharacter}, \genlift{b}}^{\modification{\genlift{b}}}(\FFuntilt).\]
\end{prop}

\begin{proof}
Let us consider the natural map $\Gr_{G, \gencocharacter}(\FFuntilt) \longrightarrow \Gr_{\genlift{G}, \genlift{\gencocharacter}}(\FFuntilt)$ induced by $f$. It suffices to show that the image of $\Gr_{G, \gencocharacter, b}^{\modification{b}}(\FFuntilt)$ lies in $\Gr_{\genlift{G}, \genlift{\gencocharacter}, \genlift{b}}^{\modification{\genlift{b}}}(\FFuntilt)$. For every $G$-bundle $\genvb$ on $\schff$, we write $\genlift{\genvb}$ for the corresponding $\genlift{G}$-bundle on $\schff$. Given an arbitrary point $x \in \Gr_{G, \gencocharacter, b}^{\modification{b}}(\FFuntilt)$, we denote by $\genlift{x}$ its image in $\Gr_{\genlift{G}}(\FFuntilt)$ and find
\[ \genvb_{\genlift{\modification{b}}} \cong \genlift{\genvb}_{\modification{b}} \simeq \genlift{\genvb}_{b, x} \cong \genvb_{\genlift{b}, \genlift{x}},\]
thereby deducing that $\genlift{x}$ lies in $\Gr_{\genlift{G}, \genlift{\gencocharacter}, \genlift{b}}^{\modification{\genlift{b}}}(\FFuntilt)$ as desired. 
\end{proof}

\begin{defn}
Let us regard elements of $\newtonset{G}$ as dominant rational cocharacters. We define the \emph{Bruhat order} $\leq$ on $\newtonset{G}$ by writing $\gencocharacter \leq \gencocharacter'$ if $\gencocharacter' - \gencocharacter$ is a linear combination of positive coroots with nonnegative coefficients. 
\end{defn}

\begin{prop}\label{functoriality of bruhat order on newton set}
Given a closed embedding $G \inj \genlift{G}$ of reductive groups over $\genlocfield$, the induced map $\newtonset{G} \longrightarrow \newtonset{\genlift{G}}$ is compatible with the Bruhat order. 
\end{prop}

\begin{proof}
The assertion is straightforward to verify by definition. 
\end{proof}

\begin{example}\label{bruhat order for classical groups}
Let us describe the Bruhat order for groups considered in Example \ref{examples of newton map and kottwitz map}. 
\begin{enumerate}[label=(\arabic*)]
\item Given two elements $\gencocharacter = (\gencocharacter_1, \cdots, \gencocharacter_n)$ and $\gencocharacter' = (\gencocharacter_1', \cdots, \gencocharacter_n')$ of $\newtonset{\GL_n}$, we have $\gencocharacter \leq \gencocharacter'$ if and only if the following equivalent conditions are satisfied:
\smallskip

\begin{enumerate}[label=(\roman*)]
\item We have inequalities
\[ \sum_{i = 1}^j \gencocharacter_i \leq \sum_{i = 1}^j \gencocharacter_i' \quad \text{ for } j = 1, \cdots, n\]
with equality for $j = n$. 
\smallskip

\item If $\gencocharacter$ and $\gencocharacter'$ are regarded as polygons on the interval $[0, n]$, then $\gencocharacter$ lies below $\gencocharacter'$ with the same endpoints. 
\end{enumerate}
\smallskip

\item Since $\newtonset{\GSp_{2n}}$ is a subset of $\newtonset{\GL_{2n}}$, Proposition \ref{functoriality of bruhat order on newton set} implies that the Bruhat order on $\newtonset{\GL_n}$ restricts to the Bruhat order on $\newtonset{\GSp_{2n}}$. 
\end{enumerate}
\end{example}

\begin{defn}\label{invariants of dom cochars}
Let $\gencocharacter$ be a dominant cocharacter of $G$.
\begin{enumerate}[label=(\arabic*)]
\item The \emph{dual} of $\gencocharacter$ is the unique dominant cocharacter $\dualpolygon{\gencocharacter}$ in the conjugacy class of $-\gencocharacter$. 
\smallskip

\item The \emph{Galois average} of $\gencocharacter$ is defined by
\[\galavg{\gencocharacter}:= \dfrac{1}{[\absgalgrp: \absgalgrp_\gencocharacter]} \sum_{\tau \in \absgalgrp/\absgalgrp_\gencocharacter} \tau(\gencocharacter) \in \newtonset{G}\]
where $\absgalgrp_\gencocharacter$ denotes the stabilizer of $\gencocharacter$ in $\absgalgrp$. 
\smallskip

\item The \emph{degree} of $\gencocharacter$, denoted by $\cochardeg{\gencocharacter}$, is the image of $\gencocharacter$ under the natural projection map $\cocharlatt \surj \pi_1(G)_\absgalgrp$. 
\end{enumerate}
\end{defn}

\begin{example}\label{cochar invariants for classical groups}
Let us illustrate Definition \ref{invariants of dom cochars} for groups considered in Example \ref{examples of newton map and kottwitz map}. 
\begin{enumerate}[label=(\arabic*)]
\item For $\GL_n$, the Galois group $\absgalgrp$ acts trivially on cocharacters as $\GL_n$ is split. Hence all dominant cocharacters of $\GL_n$ are elements of $\newtonset{\GL_n}$. Given a dominant cocharacter $\gencocharacter = (\gencocharacter_1, \cdots, \gencocharacter_n)$ of $\GL_n$, we have
\[ \dualpolygon{\gencocharacter} = (-\gencocharacter_n, \cdots, -\gencocharacter_1), \quad \galavg{\gencocharacter} = \gencocharacter = (\gencocharacter_1, \cdots, \gencocharacter_n), \quad \cochardeg{\gencocharacter} = \gencocharacter_1 + \cdots + \gencocharacter_n.\]

\item For $\GSp_{2n}$, the Galois group $\absgalgrp$ acts trivially on cocharacters as $\GSp_{2n}$ is split. Hence all dominant cocharacters of $\GSp_{2n}$ are elements of $\newtonset{\GSp_{2n}}$. Given a dominant cocharacter $\gencocharacter = (\gencocharacter_1, \cdots, \gencocharacter_{2n})$ of $\GSp_{2n}$, we have
\[ \dualpolygon{\gencocharacter} = (-\gencocharacter_{2n}, \cdots, -\gencocharacter_1), \quad \galavg{\gencocharacter} = \gencocharacter = (\gencocharacter_1, \cdots, \gencocharacter_{2n}), \quad \cochardeg{\gencocharacter} = \dfrac{1}{n} (\gencocharacter_1 + \cdots + \gencocharacter_{2n}).\]

\end{enumerate}
\end{example}

\begin{prop}[{\cite[Proposition 5.2]{CFS_admlocus}, \cite[Corollary 5.4]{Viehmann_weakadmlocNewton}}]\label{nonempty newton strata assoc to basic element}
Let $\gencocharacter$ be a dominant cocharacter of $G$. Take elements $b, \modification{b} \in B(G)$ such that $b$ is basic. The Newton stratum $\Gr_{G, \gencocharacter, b}^{\modification{b}}$ is not empty if and only if we have $\kottwitzmap[\modification{b}] = \kottwitzmap[b] - \cochardeg{\gencocharacter}$ and $\newtonmap[\modification{b}] \leq \newtonmap[b] + \galavg{(\dualpolygon{\gencocharacter})}$. 
\end{prop}

\section{Nonempty Newton strata in minuscule Schubert cells for $\GSp_{2n}$}

\begin{defn}\label{defn of id cochar, ord cochar, slopewise dominance}
Given a rational tuple $\gentuple$, we denote its $i$-th entry by $\HNslope{\gentuple}{i}$.
\begin{enumerate}[label=(\arabic*)]
%
\item 
The \emph{identity cocharacter} of $\GSp_{2n}$ is the element $\idcochar \in \newtonset{\GSp_{2n}}$ with 
\[\HNslope{\idcochar}{i} = 1 \quad \text{ for } i = 1, \cdots, 2n.\]

\item 
The \emph{ordinary cocharacter} of $\GSp_{2n}$ is the element $\ordcochar \in \newtonset{\GSp_{2n}}$ with 
\[ \HNslope{\ordcochar}{i} = \begin{cases} 1 & \text{ for } i = 1, \cdots, n, \\ 0 & \text{ for } i = n+1, \cdots, 2n. \end{cases}\]

\item We define the \emph{slopewise dominance order} $\slopedom$ on $\newtonset{\GSp_{2n}}$ by writing $\gencocharacter \slopedom \gencocharacter'$ if we have $\HNslope{\gencocharacter}{i} \leq \HNslope{\gencocharacter'}{i}$ for $i = 1, \cdots, 2n$. 
\end{enumerate}
\end{defn}

\begin{remark}
Our terminologies for $\idcochar$ and $\ordcochar$ come from the following facts:
\begin{enumerate}[label=(\alph*)]
\item As a polygon, $\idcochar$ coincides with the graph of the identity function on $[0, 2n]$. 
\smallskip

\item As a polygon, $\ordcochar$ coincides with the Newton polygon of the ordinary polarized abelian variety of dimension $2n$. 
%
\end{enumerate}
\end{remark}


\begin{prop}\label{embedding of B(GSp2n) into B(GL2n)}
The natural map $B(\GSp_{2n}) \longrightarrow B(\GL_{2n})$ is injective. 
\end{prop}

\begin{proof}
The assertion is evident by the commutative diagram
\[
\begin{tikzcd}
B(\GL_{2n}) \arrow[r, "\newtonmap", hookrightarrow] &[2em] \newtonset{\GL_{2n}} \\[0.5em]
B(\GSp_{2n}) \arrow[r, "\newtonmap", hookrightarrow] \arrow[u] & \newtonset{\GSp_{2n}} \arrow[u, hookrightarrow] 
\end{tikzcd}
\]
where the Newton maps are injective as noted in Example \ref{examples of newton map and kottwitz map}. 
\end{proof}

\begin{remark}
While Proposition \ref{embedding of B(GSp2n) into B(GL2n)} is an interesting fact, it is not essential for our argument as we will use it only for notational convenience. Every part in our argument that invokes Proposition \ref{embedding of B(GSp2n) into B(GL2n)} will rely only on the injectivity of the map $\newtonset{\GSp_{2n}} \inj \newtonset{\GL_{2n}}$. 
\end{remark}


\begin{prop}[{\cite[Lemma 3.1.4 and Proposition 3.3.2]{Hong_generalnewtonstrataGLn}}]\label{Newton stratum for GLn, conditions on nonempty strata}
Let us regard $\idcochar$ and $\ordcochar$ as dominant cocharacters of $\GL_{2n}$. Consider elements $b, \modification{b} \in B(\GL_{2n})$. 
\begin{enumerate}[label=(\arabic*)]
\item If the Newton stratum $\Gr_{\GL_{2n}, \idcochar, b}^{\modification{b}}$ contains $\genlift{x}:= \idcochar(\gendRunif) \GL_{2n}(\BdR^+) \in \Gr_{\GL_{2n}, \idcochar}(\FFuntilt)$, we have $\newtonmap[\modification{b}] = \newtonmap[b] - \idcochar$. 
\smallskip

\item If the Newton stratum $\Gr_{\GL_{2n}, \ordcochar, b}^{\modification{b}}$ is not empty, we have
\[\newtonmap[\modification{b}] \leq \newtonmap[b] + \dualpolygon{\ordcochar} \quad\text{ and } \quad
\newtonmap[\modification{b}] \slopedom \newtonmap[b] \slopedom \newtonmap[\modification{b}] + \idcochar.\]
\end{enumerate}
\end{prop}

\begin{remark}
In fact, the second statement holds for any dominant cocharacter $\gencocharacter$ of $\GL_{2n}$ whose entries are either $0$ or $1$. 
\end{remark}

\begin{lemma}\label{Newton stratum for central twist}
Given elements $b, \modification{b} \in B(\GSp_{2n})$ such that the Newton stratum $\Gr_{\GSp_{2n}, \idcochar, b}^{\modification{b}}$ contains $x:= \idcochar(\gendRunif) \GSp_{2n}(\BdR^+) \in \Gr_{\GSp_{2n}, \idcochar}(\FFuntilt)$, we have $\newtonmap[\modification{b}] = \newtonmap[b] - \idcochar$. 
\end{lemma}

\begin{proof}
We may regard $\idcochar$ as a dominant cocharacter of $\GL_{2n}$. In addition, in light of Proposition \ref{embedding of B(GSp2n) into B(GL2n)}, we may regard $b$ and $\modification{b}$ as elemnts of $B(\GL_{2n})$. By Proposition \ref{functoriality of newton strata}, we have a natural map
\[ \Gr_{\GSp_{2n}, \idcochar, b}^{\modification{b}}(\FFuntilt) \longrightarrow \Gr_{\GL_{2n}, \idcochar, b}^{\modification{b}}(\FFuntilt)\]
which sends $x$ to $\genlift{x}:= \idcochar(\gendRunif) \GL_{2n}(\BdR^+)$. Hence the desired assertion is evident by Proposition \ref{Newton stratum for GLn, conditions on nonempty strata}. 
\end{proof}

\begin{prop}\label{reduction to minimal minuscule}
Let $\gencocharacter$ be a dominant cocharacter of $\GSp_{2n}$ with nonnegative entries. 
Given arbitrary elements $b, \modification{b} \in B(\GSp_{2n})$, we have the following equivalent conditions:
\begin{enumerate}[label=(\roman*)]
\item\label{nonemptiness of original stratum} $\Gr_{\GSp_{2n}, \gencocharacter, b}^{\modification{b}}$ is nonempty. 
\smallskip

\item\label{nonemptiness of dual stratum} $\Gr_{\GSp_{2n}, \dualpolygon{\gencocharacter}, \modification{b}}^{b}$ is nonempty. 
\smallskip

\item\label{nonemptiness of shifted stratum} $\Gr_{\GSp_{2n}, \gencocharacter+\idcochar, b}^{\genlift{\modification{b}}}$ is nonempty for $\genlift{\modification{b}} \in B(\GSp_{2n})$ with $\newtonmap[\genlift{\modification{b}}] = \newtonmap[\modification{b}] - \idcochar$. 
\end{enumerate}
\end{prop}

\begin{proof}
For an arbitrary point $x = g\gencocharacter(\gendRunif)\GSp_{2n}(\BdR^+) \in \Gr_{\GSp_{2n}, \gencocharacter, b}^{\modification{b}}(\FFuntilt)$, we take 
\[\dualpolygon{x} := g^{-1}\dualpolygon{\gencocharacter}(\gendRunif)\GSp_{2n}(\BdR^+) \in \Gr_{\GSp_{2n}, \dualpolygon{\gencocharacter}}(\FFuntilt)\] 
and find $\genvb_{\modification{b}, \dualpolygon{x}} \simeq \genvb_b$, thereby deducing that $\dualpolygon{x}$ lies in $\Gr_{\GSp_{2n}, \dualpolygon{\gencocharacter}, \modification{b}}^{b}(\FFuntilt)$. Similarly, every point in $\Gr_{\GSp_{2n}, \dualpolygon{\gencocharacter}, \modification{b}}^{b}(\FFuntilt)$ yields a point in $\Gr_{\GSp_{2n}, \gencocharacter, b}^{\modification{b}}(\FFuntilt)$. Hence we establish the equivalence of the conditions \ref{nonemptiness of original stratum} and \ref{nonemptiness of dual stratum} by Proposition \ref{nonemptiness of newton strata via points}.

It remains to verify the equivalence of the conditions \ref{nonemptiness of original stratum} and \ref{nonemptiness of shifted stratum}. For an arbitrary point $x = g\gencocharacter(\gendRunif)\GSp_{2n}(\BdR^+) \in \Gr_{\GSp_{2n}, \gencocharacter, b}^{\modification{b}}(\FFuntilt)$, we take 
\[\genlift{x} := g\gencocharacter(\gendRunif)\idcochar(\gendRunif) \GSp_{2n}(\BdR^+) \in \Gr_{\GSp_{2n}, \gencocharacter+\idcochar}(\FFuntilt)\] 
and find $\genvb_{b, \genlift{x}} \simeq \genvb_{\genlift{\modification{b}}}$ by Lemma \ref{Newton stratum for central twist}, thereby deducing that $\genlift{x}$ lies in $\Gr_{\GSp_{2n}, \gencocharacter+\idcochar, b}^{\genlift{\modification{b}}}(\FFuntilt)$. Similarly, every point in $\Gr_{\GSp_{2n}, \gencocharacter+\idcochar}(\FFuntilt)$ gives rise to a point in $\Gr_{\GSp_{2n}, \gencocharacter, b}^{\modification{b}}(\FFuntilt)$. Hence we complete the proof by Proposition \ref{nonemptiness of newton strata via points}. 
\end{proof}

\begin{prop}\label{minuscule cochars of GSp2n}
For a minuscule dominant cocharacter $\gencocharacter$ of $\GSp_{2n}$, we have $\gencocharacter = d \cdot \idcochar$ or $\gencocharacter = d \cdot \idcochar + \ordcochar$ for some $d \in \Z$. 
%
%
\end{prop}

\begin{proof}
Let us identify the character lattice of $\GL_{2n}$ with $\Z^{2n}$ and write $e_1, \cdots, e_{2n}$ for its standard basis elements. Then we can identify the character lattice of $\GSp_{2n}$ as the lattice generated by
\[ e_0' := e_{n+1} + \cdots + e_{2n} \quad \text{ and } \quad e_i':= e_i - e_{2n+1-i} \text{ for } i = 1, \cdots, n. \]
The set of positive roots for $\GL_{2n}$ and $\GSp_{2n}$ are given by
\begin{align*} 
\roots_{\GL_{2n}}^+ &= \{ e_i - e_j: 1 \leq i < j \leq 2n\}, \\
\roots_{\GSp_{2n}}^+ &= \{ e_i' - e_j': 1 \leq i < j \leq n\} \cup \{ e_i' + e_j' - e_0': 1 \leq i \leq j \leq n\}.
\end{align*}
Now the assertion is straightforward to verify. 
\end{proof}


\begin{prop}\label{nonempty newton strata for GSp2n, necessity part}
Let $\gencocharacter$ be a minuscule dominant cocharacter of $\GSp_{2n}$. Assume that all entries of $\gencocharacter$ are either $d$ or $d+1$ for some $d \in \Z$. For elements $b, \modification{b} \in B(\GSp_{2n})$ such that the Newton stratum $\Gr_{\GSp_{2n}, \gencocharacter, b}^{\modification{b}}$ is not empty, we have
\[\newtonmap[\modification{b}] \leq \newtonmap[b] + \dualpolygon{\gencocharacter} \quad\text{ and } \quad
\newtonmap[\modification{b}] + d \cdot \idcochar \slopedom \newtonmap[b] \slopedom \newtonmap[\modification{b}] + (d+1) \cdot \idcochar.\]
\end{prop}

\begin{proof}
By Propositions \ref{reduction to minimal minuscule} and \ref{minuscule cochars of GSp2n}, we may take $d = 0$ and assume that $\gencocharacter$ is either $0$ or $\ordcochar$. For $\gencocharacter = 0$, the assertion is evident since the only nonempty Newton stratum in $\Gr_{\GSp_{2n}, 0}$ with respect to $b$ is $\Gr_{\GSp_{2n}, 0, b}^b$. For $\gencocharacter = \ordcochar$, we have a natural map 
\[ \Gr_{\GSp_{2n}, \ordcochar, b}^{\modification{b}}(\FFuntilt) \longrightarrow \Gr_{\GL_{2n}, \ordcochar, b}^{\modification{b}}(\FFuntilt)\]
given by Proposition \ref{functoriality of newton strata}, where for $\Gr_{\GL_{2n}, \ordcochar, b}^{\modification{b}}$ we regard $\ordcochar$ as a dominant cocharacter of $\GL_{2n}$ and $b, \modification{b}$ as elements of $B(\GL_{2n})$ in light of Proposition \ref{embedding of B(GSp2n) into B(GL2n)}, and thus deduce the desired assertion by Propositions \ref{nonemptiness of newton strata via points} and \ref{Newton stratum for GLn, conditions on nonempty strata}.
\end{proof}

\begin{defn}
Let $\genpartition = (\genpartition_1, \cdots, \genpartition_l)$ be an ordered partition of an integer $m \leq n$. 
\begin{enumerate}[label=(\arabic*)]
\item We define the associated Levi subgroup $\genlevi_\genpartition$ of $\GSp_{2n}$ to be the group of block diagonal matrices 
\[ 
\quad\quad\quad
\begin{pmatrix}
g_1 & & & & & & \\
 & \ddots & & & & & \\
 & & g_{l} & & & & \\
 & & & h & & &\\
 & & & & \simcochar(h) (g_{l}^T)^{-1} & & \\
 & & & & & \ddots & \\
 & & & & & & \simcochar(h) (g_{1}^T)^{-1}
\end{pmatrix}
\text{ with } g_i \in \GL_{\genpartition_i}, h \in \GSp_{2(n-m)},
\]
where $\simcochar$ denotes the similitude character of $\GSp_{2(n-m)}$. 
\smallskip

\item Given an element $b \in B(\GSp_{2n})$, its \emph{reduction to $\genlevi_\genpartition$} is an element $\genred{b} \in B(\genlevi_\genpartition)$ whose image under the natural map $B(\genlevi_\genpartition) \longrightarrow B(\GSp_{2n})$ is equal to $b$. 
\end{enumerate}
\end{defn}

\begin{remark}
It is not hard to show that every standard Levi subgroup of $\GSp_{2n}$ is equal to $\genlevi_\genpartition$ for some ordered partition $\genpartition$ of an integer $m \leq n$. 
\end{remark}

\begin{prop}\label{newton set of levi for GSp2n}
Let $\genpartition = (\genpartition_1, \cdots, \genpartition_l)$ be an ordered partition of an integer $m \leq n$. We write $n_i := \genpartition_1 + \cdots + \genpartition_i$ for $i = 1, \cdots, l$. In addition, we set $n_0 := 0$ and $n_{l+1}:= n$. 
\begin{enumerate}[label=(\arabic*)]
%
\item The Newton set of $\genlevi_\genpartition$ is the set of dominant rational cocharacters of $\genlevi_\genpartition$, which is also identified with the set of tuples $(\genslope_1, \cdots, \genslope_{2n}) \in \Q^{2n}$ satisfying the following properties:
\smallskip
\begin{enumerate}[label=(\roman*)]
\item We have $\genslope_{n_{i-1}+1} \geq \cdots \geq \genslope_{n_i}$ for $i = 1, \cdots, l+1$. 
\smallskip

\item There exists a rational number $d$ with $\genslope_i + \genslope_{2n+1-i} = d$ for $i = 1, \cdots, n$. 
\end{enumerate}
\smallskip

\item Given two elements $\gencocharacter = (\gencocharacter_1, \cdots, \gencocharacter_{2n})$ and $\gencocharacter' = (\gencocharacter_1', \cdots, \gencocharacter_{2n}')$ of $\newtonset{\genlevi_\genpartition}$, the inequality $\gencocharacter \leq \gencocharacter'$ holds if and only if we have inequalities
\[ \sum_{i = 1}^j \gencocharacter_i \leq \sum_{i = 1}^j \gencocharacter_i' \quad \text{ for } j = 1, \cdots, 2n\]
with equalities for $j = n_1, \cdots, n_{l+1}$. 
\smallskip

\item There exists a natural isomorphism $\pi_1(\genlevi_\genpartition)_\absgalgrp \cong \Z^{l+1}$. 
\smallskip

\item Given a dominant cocharacter $\gencocharacter = (\gencocharacter_1, \cdots, \gencocharacter_{2n})$ of $\genlevi_\genpartition$, its dual $\dualpolygon{\gencocharacter} = (\dualpolygon{\gencocharacter}_1, \cdots, \dualpolygon{\gencocharacter}_{2n})$ is specified by the following identities:
\smallskip
\begin{itemize}
\item $\dualpolygon{\gencocharacter}_{n_{i-1}+j} = -{\gencocharacter}_{n_i+1-j}$ for $i = 1, \cdots, l$ and $j = 1, \cdots, \genpartition_i$. 
\smallskip

\item $\dualpolygon{\gencocharacter}_{m+i} = -{\gencocharacter}_{2n+1-m-i}$ for $i = 1, \cdots, n-m$. 
\smallskip

\item $\dualpolygon{\gencocharacter}_i + \dualpolygon{\gencocharacter}_{2n+1-i} = - (\gencocharacter_i + \gencocharacter_{2n+1-i})$ for $i = 1, \cdots, n$. 
\end{itemize}

\item Given a dominant cocharacter $\gencocharacter = (\gencocharacter_1, \cdots, \gencocharacter_{2n})$ of $\genlevi_\genpartition$, its degree $\cochardeg{\gencocharacter} = (\cochardeg{\gencocharacter}_1, \cdots, \cochardeg{\gencocharacter}_{l+1})$ is given by
\[ \cochardeg{\gencocharacter}_i = \begin{cases} \gencocharacter_{n_{i-1}+1} + \cdots + \gencocharacter_{n_i} & \text{ for } i = 1, \cdots, l\\[1em]  \dfrac{\gencocharacter_{m+1} + \cdots \gencocharacter_{2n-m}}{n-m} & \text{ for } i = l+1\end{cases}.\]
\end{enumerate}
\end{prop}

\begin{proof}
All statements are straightforward to verify by our discussion in Examples \ref{examples of newton map and kottwitz map}, \ref{bruhat order for classical groups}, and \ref{cochar invariants for classical groups}.
\end{proof}

\begin{prop}\label{Levi reduction for B(GSp2n)}
Let $\genpartition = (\genpartition_1, \cdots, \genpartition_l)$ be an ordered partition of an integer $m \leq n$. We write $n_i = \genpartition_1 + \cdots + \genpartition_i$ for $i = 1, \cdots, l$. In addition, we set $n_0: = 0$ and $n_{l+1}:=n$. 
%
Given an element $b \in B(\GSp_{2n})$ such that $\newtonmap[b]$ has integer points with $n_1, \cdots, n_l$ as $x$-coordinates, there exists a reduction of $b$ to $\genlevi_\genpartition$.
\end{prop}

\begin{proof}
For a rational tuple $\gentuple$, we denote its $i$-th entry by $\HNslope{\gentuple}{i}$. 
Our discussion in Example \ref{examples of newton map and kottwitz map} implies that there exists an integer $d$ with 
\begin{equation}\label{symmetry for slopes of newton polygon, levi reduction lemma}
\HNslope{\newtonmap[b]}{i} + \HNslope{\newtonmap[b]}{2n+1-i} = d \quad \text{ for } i = 1, \cdots, n.
\end{equation}
For each $i = 1, \cdots, l$, we define $\genpartition_i$-tuples $\newtonmaplevipart{i}$ and $\newtonmaplevipartdual{i}$ by
\[ \HNslope{\newtonmaplevipart{i}}{j} = \HNslope{\newtonmap[b]}{n_{i-1}+j} \quad \text{ and } \quad \HNslope{\newtonmaplevipartdual{i}}{j} = \HNslope{\newtonmap[b]}{2n-n_i+j} \quad \text{ for } j = 1, \cdots, \genpartition_j.\]
In addition, we define a $2(n-m)$-tuple $\newtonmaplevipartmid$ by
\[ \HNslope{\newtonmaplevipartmid}{i} = \HNslope{\newtonmap[b]}{m+i} \quad \text{ for } i = 1, \cdots, 2(n-m).\]
As polygons, these tuples form a partition of $\newtonmap[b]$ as illustrated in Figure \ref{partition of newton polygon by levi}. 
\begin{figure}[H]
\begin{tikzpicture}[scale=0.5]	
		\coordinate (left) at (0, 0);
		\coordinate (p1) at (1, 2);
		\coordinate (p2) at (2, 3);
		\coordinate (p3) at (4, 5);
		\coordinate (p4) at (6, 6);
		\coordinate (p5) at (8, 6);
		\coordinate (p6) at (9, 6);
		\coordinate (right) at (10, 5);

		\draw[step=1cm,thick, color=blue] (left) -- (p1) --  (p2) -- (p3) -- (p4) -- (p5) -- (p6) -- (right);

		\draw [fill] (left) circle [radius=0.08];
		\draw [fill] (right) circle [radius=0.08];		

		\draw [fill] (p1) circle [radius=0.08];		
		\draw [fill] (p2) circle [radius=0.08];		
		\draw [fill] (p3) circle [radius=0.08];		
		\draw [fill] (p4) circle [radius=0.08];		
		\draw [fill] (p5) circle [radius=0.08];		
		\draw [fill] (p6) circle [radius=0.08];		
		
		\draw[step=1cm,dotted] (2, -0.4) -- (2, 6.8);
		\draw[step=1cm,dotted] (4, -0.4) -- (4, 6.8);
		\draw[step=1cm,dotted] (6, -0.4) -- (6, 6.8);
		\draw[step=1cm,dotted] (8, -0.4) -- (8, 6.8);
		
		\node at (2,-0.8) {\scriptsize $n_1$};
		\node at (4,-0.8) {\scriptsize $n_2$};

		\path (p6) ++(1.2, 0) node {\color{blue}$\newtonmap[b]$};
\end{tikzpicture}
\hspace{0.3cm}
\begin{tikzpicture}[scale=0.4]
        \pgfmathsetmacro{\textycoordinate}{5}
		\draw[->, line width=0.6pt] (0, \textycoordinate) -- (1.5,\textycoordinate);
		\draw (0,0) circle [radius=0.00];	
\end{tikzpicture}
\hspace{0.3cm}
\begin{tikzpicture}[scale=0.5]
		\coordinate (left) at (0, 0);
		\coordinate (p1) at (1, 2);
		\coordinate (p2) at (2, 3);
		\coordinate (p3) at (4, 5);
		\coordinate (p4) at (6, 6);
		\coordinate (p5) at (8, 6);
		\coordinate (p6) at (9, 6);
		\coordinate (right) at (10, 5);

		\draw[step=1cm,thick, color=red] (left) -- (p1) --  (p2);
		\draw[step=1cm,thick, color=orange] (p2) -- (p3);
		\draw[step=1cm,thick, color=green] (p3) --  (p4);
		\draw[step=1cm,thick, color=magenta] (p4) --  (p5);
		\draw[step=1cm,thick, color=violet] (p5) -- (p6) --  (right);

		\draw [fill] (left) circle [radius=0.08];
		\draw [fill] (right) circle [radius=0.08];		

		\draw [fill] (p1) circle [radius=0.08];		
		\draw [fill] (p2) circle [radius=0.08];		
		\draw [fill] (p3) circle [radius=0.08];		
		\draw [fill] (p4) circle [radius=0.08];		
		\draw [fill] (p5) circle [radius=0.08];		
		\draw [fill] (p6) circle [radius=0.08];	
		
		\draw[step=1cm,dotted] (2, -0.4) -- (2, 6.8);
		\draw[step=1cm,dotted] (4, -0.4) -- (4, 6.8);
		\draw[step=1cm,dotted] (6, -0.4) -- (6, 6.8);
		\draw[step=1cm,dotted] (8, -0.4) -- (8, 6.8);
		
		\node at (2,-0.8) {\scriptsize $n_1$};
		\node at (4,-0.8) {\scriptsize $n_2$};
		
		\path (p1) ++(-1, 0) node {\color{red}$\newtonmaplevipart{1}$};
		\path (p2) ++(1, 1.8) node {\color{orange}$\newtonmaplevipart{2}$};
		\path (p3) ++(1, 1) node {\color{green}$\newtonmaplevipartmid$};
		\path (p4) ++(1.1, 0.6) node {\color{magenta}$\newtonmaplevipartdual{2}$};
		\path (p6) ++(1.2, 0) node {\color{violet}$\newtonmaplevipartdual{1}$};
\end{tikzpicture}
\caption{Partition of $\newtonmap[b]$ by integer points}\label{partition of newton polygon by levi}
\end{figure}
By construction, these polygons all have integer breakpoints. In addition, by \eqref{symmetry for slopes of newton polygon, levi reduction lemma} we find 
\begin{equation}\label{symmetry for slopes of newton polygon, levi reduction lemma mid part}
\HNslope{\newtonmaplevipartmid}{i} + \HNslope{\newtonmaplevipartmid}{2n-2m+1-i} = d \quad \text{ for } i = 1, \cdots, 2(n-m).
\end{equation}
Hence by our discussion in Example \ref{examples of newton map and kottwitz map}, we find an element $b_i \in B(\GL_{\genpartition_i})$ with $\newtonmap[b_i] = \newtonmaplevipart{i}$ for $i = 1, \cdots, l$ and also an element $c \in B(\GSp_{2(n-m)})$ with $\newtonmap[c] = \newtonmaplevipartmid$. Let us choose a representative $g_i \in \GL_{\genpartition_i}(\compmaxunram{\genlocfield})$ of $b_i$ for each $i = 1, \cdots, l$ and  a representative $h \in \GSp_{2(n-m)}(\compmaxunram{\genlocfield})$ of $c$. We take $\genred{b} \in B(\genlevi_\genpartition)$ to be the $\frobfieldaut$-conjugacy class of
\[
\begin{pmatrix}
g_1 & & & & & & \\
 & \ddots & & & & & \\
 & & g_{l} & & & & \\
 & & & h & & &\\
 & & & & \simcochar(h) (g_{l}^T)^{-1} & & \\
 & & & & & \ddots & \\
 & & & & & & \simcochar(h) (g_{1}^T)^{-1}
\end{pmatrix}
\in \genlevi_{\genpartition}(\compmaxunram{\genlocfield})
\]
where $\simcochar$ denotes the similitude character of $\GSp_{2(n-m)}$, and denote by $\levilift{b}$ the image of $\genred{b}$ under the natural map $B(\genlevi_\genpartition) \longrightarrow B(\GSp_{2n})$. By the functoriality of the Newton map, $\newtonmap[\levilift{b}]$ is given by the image of $\newtonmap[\genred{b}]$ under the natural map $\newtonset{\genlevi_\genpartition} \longrightarrow \newtonset{\GSp_{2n}}$. It is then not hard to see by \eqref{symmetry for slopes of newton polygon, levi reduction lemma} and \eqref{symmetry for slopes of newton polygon, levi reduction lemma mid part} that $\newtonmap[\levilift{b}]$ coincides with $\newtonmap[b]$. Since the Newton map for $\GSp_{2n}$ is injective as noted in Example \ref{examples of newton map and kottwitz map}, we find $b = \levilift{b}$ and consequently complete the proof. 
\end{proof}

\begin{theorem}\label{classification of nonempty Newton strata for GSp2n}
Let $\gencocharacter$ be a minuscule dominant cocharacter of $\GSp_{2n}$ whose entries are either $d$ or $d+1$ for some $d \in \Z$. Consider elements $b, \modification{b} \in B(\GSp_{2n})$ such that any two distinct slopes in $\newtonmap[b]$ differ by more than $1$.
The Newton stratum $\Gr_{\GSp_{2n}, \gencocharacter, b}^{\modification{b}}$ is nonempty if and only if the following conditions are satisfied:
\begin{enumerate}[label=(\roman*)]
\item\label{newton polygon inequalities for nonempty Newton strata} We have 
$\newtonmap[\modification{b}] \leq \newtonmap[b] + \dualpolygon{\gencocharacter}$ and $\newtonmap[\modification{b}] + d \cdot \idcochar \slopedom \newtonmap[b] \slopedom \newtonmap[\modification{b}] + (d+1) \cdot \idcochar$. 
\smallskip

\item\label{breakpoint condition for nonempty Newton strata} For each breakpoint of $\newtonmap[b]$, there exists 
a breakpoint of $\newtonmap[\modification{b}]$ with the same $x$-coordinate. 
\end{enumerate}
\end{theorem}
\begin{figure}[H]
\begin{tikzpicture}[scale=0.7]
		\coordinate (p00) at (1.5, 4);
		\coordinate (p11) at (2.8, 6.3);
		\coordinate (r11) at (4, 7.8);
		\coordinate (p22) at (5.2, 8.7);
		\coordinate (p33) at (6.5, 9);
		\coordinate (p44) at (8, 8);

		\coordinate (left) at (0, 0);
		\coordinate (q0) at (1.5, 3.2);
		\coordinate (q1) at (4, 5.2);
		\coordinate (q2) at (6.5, 5.7);
		\coordinate (q3) at (8, 4);
		
		\coordinate (r1) at (4, 3.8);
		\coordinate (r2) at (6.5, 3.5);

		\coordinate (p0) at (1.5, 2.5);
		\coordinate (p1) at (2.8, 3.5);
		\coordinate (p2) at (5.2, 3.5);
		\coordinate (p3) at (6.5, 2.5);
		\coordinate (p4) at (8, 0);
				
		\draw[step=1cm,thick, color=red] (left) -- (q0) --  (q1) -- (q2) -- (q3);
		\draw[step=1cm,thick, color=green] (left) -- (p0) --  (p1) -- (r1) -- (p2) -- (p3) -- (p4);
		\draw[step=1cm,thick, color=teal] (left) -- (p00) --  (p11) -- (r11) -- (p22) -- (p33) -- (p44);
		\draw[step=1cm,thick, color=blue] (q1) -- (r2) -- (p4);
		
		\draw [fill] (q0) circle [radius=0.05];		
		\draw [fill] (q1) circle [radius=0.05];		
		\draw [fill] (q2) circle [radius=0.05];		
		\draw [fill] (q3) circle [radius=0.05];		
		\draw [fill] (left) circle [radius=0.05];
		
		\draw [fill] (r1) circle [radius=0.05];		
		\draw [fill] (r2) circle [radius=0.05];	
		
		\draw [fill] (p0) circle [radius=0.05];		
		\draw [fill] (p1) circle [radius=0.05];		
		\draw [fill] (p2) circle [radius=0.05];		
		\draw [fill] (p3) circle [radius=0.05];		
		\draw [fill] (p4) circle [radius=0.05];		

		\draw [fill] (p00) circle [radius=0.05];		
		\draw [fill] (p11) circle [radius=0.05];		
		\draw [fill] (r11) circle [radius=0.05];		
		\draw [fill] (p22) circle [radius=0.05];	
		\draw [fill] (p33) circle [radius=0.05];	
		\draw [fill] (p44) circle [radius=0.05];	
		
		\draw[step=1cm,dotted] (1.5, -0.4) -- (1.5, 9.2);
		\draw[step=1cm,dotted] (4, -0.4) -- (4, 9.2);
		\draw[step=1cm,dotted] (6.5, -0.4) -- (6.5, 9.2);

		
		\path (q1) ++(0.8, 0.6) node {\color{red}$\newtonmap[b]$};
		\path (r2) ++(1.8, -0.8) node {\color{blue}$\newtonmap[b] + \dualpolygon{\gencocharacter}$};
		\path (p00) ++(-1.3, 0) node {\color{teal}$\newtonmap[\modification{b}] + \idcochar$};
		\path (p0) ++(0.7, -0.3) node {\color{green}$\newtonmap[\modification{b}]$};

\end{tikzpicture}
\caption{Illustration of the conditions in Theorem \ref{classification of nonempty Newton strata for GSp2n} for $d =0$}
\end{figure}

\begin{proof}
For a rational tuple $\gentuple$, we denote its $i$-th entry by $\HNslope{\gentuple}{i}$. By Propositions \ref{reduction to minimal minuscule} and \ref{minuscule cochars of GSp2n}, we may set $d = 0$ and assume that $\gencocharacter$ is either $0$ or $\ordcochar$. For $\gencocharacter = 0$, the assertion is evident since the only nonempty Newton stratum in $\Gr_{\GSp_{2n}, 0}$ with respect to $b$ is $\Gr_{\GSp_{2n}, 0, b}^b$. Therefore we may henceforth take $\gencocharacter = \ordcochar$.

Let us first assume that $\Gr_{\GSp_{2n}, \gencocharacter, b}^{\modification{b}}$ is nonempty. The necessity of the condition \ref{newton polygon inequalities for nonempty Newton strata} immediately follows from Proposition \ref{nonempty newton strata for GSp2n, necessity part}. Now we take an arbitrary breakpoint of $\newtonmap[b]$ and denote its $x$-coordinate by $m$. By the condition \ref{newton polygon inequalities for nonempty Newton strata} and the assumption on slopes of $\newtonmap[b]$, we find
\begin{equation*} 
\HNslope{\newtonmap[\modification{b}]}{m+1} \leq \HNslope{\newtonmap[b]}{m+1} < \HNslope{\newtonmap[b]}{m} -1 \leq \HNslope{\newtonmap[\modification{b}]}{m}
\end{equation*}
and consequently deduce that $\modification{\gencocharacter}$ has a breakpoint with $x$-coordinate $m$, thereby establishing the necessity of the condition \ref{breakpoint condition for nonempty Newton strata}.

For the converse, we now assume that the conditions \ref{newton polygon inequalities for nonempty Newton strata} and \ref{breakpoint condition for nonempty Newton strata} are satisfied. We consider the breakpoints of $\newtonmap[b]$ on the interval $[0, n]$ and denote their $x$-coordinates by $n_1, \cdots, n_l$ in ascending order. We also set $n_0: = 0$ and take $\genpartition = (\genpartition_1, \cdots, \genpartition_l)$ to be an ordered partition of $m = n_l$ with 
\[\genpartition_i:= n_i - n_{i-1} \quad \text{ for } i = 1, \cdots, l.\] 
It is not hard to see that the centralizer of $\newtonmap[b]$ is $\genlevi_\genpartition$. Moreover, Proposition \ref{Levi reduction for B(GSp2n)} and the condition \ref{breakpoint condition for nonempty Newton strata} together imply that $b$ and $\modification{b}$ admit reductions to $\genlevi_\genpartition$, which we respectively denote by $\genred{b}$ and $\genred{\modification{b}}$. Let us now set
\[d_i:= \sum_{j = n_{i-1}+1}^{n_i} \HNslope{\newtonmap[b]}{j} - \sum_{j=n_{i-1}+1}^{n_i} \HNslope{\newtonmap[\modification{b}]}{j} \quad \text{ for } i = 1, \cdots, l.\]
By the condition \ref{newton polygon inequalities for nonempty Newton strata}, we have $0 \leq d_i \leq \genpartition_i$ for $i = 1, \cdots, l$. Take $\genred{\gencocharacter}$ to be the $2n$-tuple specified by the following properties:
\begin{enumerate}[label=(\alph*)]
\item For $i = 1, \cdots, l+1$ and $j = 1, \cdots, \genpartition_i$, we have
\[\HNslope{\genred{\gencocharacter}}{n_{i-1}+j} = \begin{cases} 1 & \text{ if } j \leq d_i,\\ 0 & \text { if } j > d_i,\end{cases}\]
where we set $d_{l+1}:= n- n_l = n-m$. 
\smallskip

\item For $j = 1, \cdots, n$, we have $\HNslope{\genred{\gencocharacter}}{j} + \HNslope{\genred{\gencocharacter}}{2n+1-j} = 1$. 
\end{enumerate}
It is not hard to see by Proposition \ref{newton set of levi for GSp2n} that $\genred{\gencocharacter}$ is a dominant cocharacter of $\genlevi_\genpartition$ with
\[\kottwitzmap[\genred{\modification{b}}] = \kottwitzmap[\genred{b}] - \cochardeg{\genred{\gencocharacter}} \quad \text{ and } \quad \newtonmap[\genred{\modification{b}}] \leq \newtonmap[\genred{b}] - \galavg{(\dualpolygon{\genred{\gencocharacter}})}.\]
Moreover, $\genred{b}$ is basic in $B(\genlevi_\genpartition)$ by construction. Hence Proposition \ref{nonempty newton strata assoc to basic element} implies that the Newton stratum $\Gr_{\genlevi_\genpartition, \genred{\gencocharacter}, \genred{b}}^{\genred{\modification{b}}}$ is not empty. Since we have a natural map
\[ \Gr_{\genlevi_\genpartition, \genred{\gencocharacter}, \genred{b}}^{\genred{\modification{b}}}(\FFuntilt) \longrightarrow \Gr_{\GSp_{2n}, \gencocharacter, b}^{\modification{b}}(\FFuntilt)\]
by Proposition \ref{functoriality of newton strata}, we deduce by Proposition \ref{nonemptiness of newton strata via points} that $\Gr_{\GSp_{2n}, \gencocharacter, b}^{\modification{b}}$ is not empty as desired, thereby completing the proof. 
\end{proof}

\bibliographystyle{amsalpha}

\bibliography{Bibliography}
	
\end{document}